\documentclass[11 pt]{article}

\usepackage[utf8]{inputenc}
\usepackage[english]{babel}   
\usepackage{graphicx}
\usepackage{lmodern}
\usepackage{amsmath}
\usepackage{tikz}
\usetikzlibrary{calc}
\usepackage{longtable}
\usepackage{epsfig}
\usepackage{color}
\usepackage{setspace}
\usepackage[top=2cm,bottom=2.5cm,left=2.5cm,right=2.5cm]{geometry}
\usepackage{sectsty}
\usepackage{fancyhdr}
\usepackage{wasysym}
\usepackage[colorlinks=true]{hyperref}
\usepackage{caption}
\usepackage{subcaption}
\usetikzlibrary{patterns}
\usepackage{amsfonts}
\usepackage{amssymb}
\usepackage{amsthm}

\newtheorem{theorem}{Theorem} 

\newtheorem{lemma}[theorem]{Lemma}
\newtheorem{proposition}[theorem]{Proposition}
\newtheorem{remark}[theorem]{Remark}
\newtheorem{hyp}[theorem]{Hypothesis}
\newtheorem{coro}[theorem]{Corollary}

\newcommand{\G}{\mathcal{G}}

\graphicspath{{pics/}}

\title{Computation of Miura surfaces with gradient Dirichlet boundary conditions}
\date{}
\author{\begin{minipage}{\textwidth}\centering
		Frédéric Marazzato \\
		\small{Department of Mathematics, The University of Arizona, Tucson, AZ 85721-0089, USA}\\
   \small{email: \texttt{marazzato@arizona.edu}\\}
   \end{minipage}
   }
      
\begin{document}
\hypersetup{urlcolor=blue,linkcolor=red,citecolor=blue}

\maketitle

\begin{abstract}
Miura surfaces are the solutions of a constrained nonlinear elliptic system of equations.
This system is derived by homogenization from the Miura fold, which is a type of origami fold with multiple applications in engineering.
A previous inquiry, gave suboptimal conditions for existence of solutions and proposed an $H^2$-conformal finite element method to approximate them.
In this paper, the existence of Miura surfaces is studied using a gradient formulation.
It is also proved that, under some hypotheses, the constraints propagate from the boundary to the interior of the domain.
Then, a numerical method based on a stabilized least-square formulation, conforming finite elements and a Newton method is introduced to approximate Miura surfaces.
The numerical method is proved to converge and numerical tests are performed to demonstrate its robustness.
\end{abstract}

\textit{Keywords: Origami, nonlinear elliptic equation, Kinematics of deformation.}

\textit{AMS Subject Classification: 35J66, 65N12, 65N15, 65N30 and 74A05.}

\section{Introduction}
The ancient art of origami has attracted a lot of attention in recent years.
The monograph \cite{lang2017twists} presents a very complete description of many different types of origami patterns as well as some questions surrounding them.
Origami have found many applications in engineering.
Origami folds are, for instance, used to ensure that airbags inflate correctly from their folded state \cite{badagavi2017use}.
Applications in aerospace engineering are as varied as radiators, panels for radio telescopes \cite{morgan2016approach} or soft robots \cite{rafsanjani2019programming}.
More recently, origami have been studied with a view to produce metamaterials \cite{overvelde2016three,boatti2017origami,wickeler2020novel}.
This article focuses on the Miura fold, or Miura ori, which was introduced in \cite{miura1969proposition}.
The Miura ori is flat when completely unfolded and can be fully folded into a very compact form. 
However, the Miura fold can also assume different shapes when partially unfolded.
\cite{schenk2013geometry} and \cite{wei2013geometric} provided a description of these partially unfolded states through a computation of the Poisson ratio of the Miura fold, which happens to be negative for in-plane deformations and positive for out-of-plane bending.
Miura folds are therefore globally saddle shaped.
A homogenization procedure for origami folds was proposed in \cite{nassar2017macroscopic} and \cite{nassar2017curvature} and then applied to the Miura fold in \cite{lebee2018fitting}.
The authors obtained a constrained nonlinear elliptic equation describing the limit surface, which is called a Miura surface.
The constraints are both equality and inequality constraints.
In \cite{marazzato}, existence and uniqueness of solutions was proved but only for the unconstrained equation and under restrictive assumptions.
\cite{marazzato} also proposed an $H^2$-conformal Finite Element Method (FEM) to compute solutions to the problem, which is computationally involved.

In this paper, the existence of solutions to the constrained equation is proved under adequate boundary conditions.
Thereafter, a stabilized least-squares formulation is proposed, coupled to a Newton method and $\mathbb{P}^1$--Lagrange finite elements to approximate the solutions.
The robustness of the method is demonstrated on several test cases.

\section{Continuous equations}
\label{sec:continuous}
\subsection{Modeling of the Miura fold}
The Miura fold is based on the reference cell sketched in Figure \ref{fig:Miura cell}, in which all edges have unit length.
\begin{figure}
\centering
\begin{tikzpicture} [scale=1.5]
\pgfmathsetmacro{\r}{1}
\coordinate (a) at (0,0);
\coordinate (b) at (\r,0);
\coordinate (c) at (2*\r,0);
\coordinate (d) at \r*(0.5,{sqrt(2)/2});
\coordinate (e) at ($(d)+(\r,0)$);
\coordinate (f) at ($(e)+(\r,0)$);
\coordinate (g) at \r*(0.5,{-sqrt(2)/2});
\coordinate (h) at ($(g)+(\r,0)$);
\coordinate (i) at ($(h)+(\r,0)$);

\draw[-] (a) -- (b);
\draw[-] (c) -- (b);
\draw[-] (a) -- (d);
\draw[-] (e) -- (d);
\draw[-] (e) -- (f);
\draw[-] (c) -- (f);
\draw[-] (b) -- (e);
\draw[-] (c) -- (i);
\draw[-] (a) -- (g);
\draw[-] (g) -- (h);
\draw[-] (h) -- (i);
\draw[-] (b) -- (h);

\draw[-,dashed] (d) -- (b);
\draw[-,dashed] (e) -- (c);
\draw[-,dashed] (c) -- (h);
\draw[-,dashed] (g) -- (b);

\end{tikzpicture}
\caption{Miura reference cell.}
\label{fig:Miura cell}
\end{figure}
The reference cell is made of four parallelograms and can be folded along the full lines in Figure \ref{fig:Miura cell}.
Note that the cell's parallelograms are not allowed to stretch.
Following \cite{wei2013geometric,schenk2013geometry}, we consider that the {parallelograms} can also bend along the dashed lines, which at order one, is akin to folding along them.
We restrict the deformations of the Miura cell to folding along crease lines or bending in the facets.
A Miura tessellation is based on the continuous juxtaposition of reference cells, dilated by a factor $r > 0$.
In the spirit of homogenization, \cite{nassar2017curvature,lebee2018fitting} have proposed a procedure to compute a surface that is the limit when $r \to 0$ of a Miura tessellation.
This procedure leads to the constrained PDE presented in the next section.

\subsection{Strong form equations}
Let $\Omega \subset \mathbb{R}^2$ be a bounded convex polygon that can be perfectly fitted by triangular meshes.
Note that, due to the convexity hypothesis, the boundary $\partial \Omega$ is Lipschitz \cite{grisvard2011elliptic}.
We denote by $n \in \mathbb{R}^2$ the unit normal to $\partial \Omega$.
Let $\varphi : \Omega \subset \mathbb{R}^2 \rightarrow \mathbb{R}^3$ be a parametrization of the homogenized surface constructed from a Miura tessellation.
As shown in \cite{lebee2018fitting}, $\varphi$ is a solution of the following strong form equation:
\begin{subequations}
\label{eq:strong form equations} 
\begin{equation}
\label{eq:min surface eq}
p(\varphi_x) \varphi_{xx} + q(\varphi_y) \varphi_{yy} = 0 \in \mathbb{R}^3,
\end{equation}
\begin{equation}
\label{eq:microstructure}
\left(1 - \frac14 |\varphi_x|^2 \right) |\varphi_y|^2 = 1,
\end{equation}
\begin{equation}
\label{eq:local basis}
\varphi_{x} \cdot \varphi_{y}  = 0,
\end{equation}
\begin{equation}
\label{eq:ineq constraint}
0 < |\varphi_{x}|^2 \le 3, \quad 1 <  |\varphi_{y}|^2 \le 4,
\end{equation}
\end{subequations}
where 
\[p(\varphi_x) = \frac4{4 - |\varphi_{x}|^2}, \quad q(\varphi_y) = \frac4{|\varphi_{y}|^2}, \]
and the subscripts $x$ and $y$ stand respectively for $\partial_x$ and $\partial_y$.
\eqref{eq:min surface eq} is a nonlinear elliptic equation that describes the fact that the in-plane and out-of-plane Poisson ratios for the bending mode of the Miura fold are equal, see \cite{schenk2013geometry,wei2013geometric}.
\eqref{eq:microstructure} and \eqref{eq:local basis} are equality constraints that stem from the description of the metric tensor of a Miura surface,
see \cite{lebee2018fitting}.
Finally, \eqref{eq:ineq constraint} is enforced to prevent the metric tensor of the surface $\varphi(\Omega)$ from being singular.
In practice, that means that, locally, when the bounds are reached, the pattern is fully folded or fully unfolded.
Note that the constraint $1 < |\varphi_y|^2$ in \eqref{eq:ineq constraint} is particularly challenging as it is non-convex.

The paper \cite{marazzato} proved that there exists solutions to \eqref{eq:min surface eq} but it also showed that some solutions of \eqref{eq:min surface eq} do not verify the constraints \eqref{eq:microstructure}-\eqref{eq:ineq constraint}.
As stated in \cite{lebee2018fitting}, these constraints are necessary to construct a Miura surface.
The main result of this paper is Theorem \ref{th:general} which proves the existence of solutions to \eqref{eq:strong form equations} under some hypothesis and appropriate boundary conditions described in the following.

\subsection{Continuous setting}
We introduce the Hilbert space $V:=\left(H^2(\Omega)\right)^3$, equipped with the usual $\left(H^2(\Omega)\right)^3$ Sobolev norm.
Note that due to Rellich--Kondrachov theorem \cite[Theorem 9.16]{brezis}, $V \subset \left(\mathcal{C}^0(\bar{\Omega}) \right)^3$.
For $\varphi \in V$, let $\mathcal{A}(\varphi):V \mapsto L^2(\Omega)^3$, be the operator defined for $\psi \in V$ as
\begin{equation}
\label{eq:inner operator}
\mathcal{A}(\varphi) \psi := p(\varphi_x) \psi_{xx} + q(\varphi_y) \psi_{yy} \in \mathbb{R}^3.
\end{equation}
Solving Equation \eqref{eq:min surface eq} consists in finding $\varphi \in V$ such that
\begin{equation}
\label{eq:nonlinear operator eq}
\mathcal{A}(\varphi)\varphi = 0 \text{ in } \Omega.
\end{equation}
The operator $\mathcal{A}(\varphi)$ is not well adapted to obtain a constrained solution of \eqref{eq:min surface eq}, see \cite{marazzato}.
Equation \eqref{eq:nonlinear operator eq} is thus reformulated into an equation on $\mathcal{G} \equiv \nabla \varphi$.
Let $W := H^1(\Omega)^{3 \times 2}$ equipped with the usual $H^1(\Omega)$ norm, and $W_0 := H^1_0(\Omega)^{3 \times 2}$.
For $G \in W$, we write $G=(G^x,G^y)$, where $G^x,G^y \in \mathbb{R}^3$.
Let $A(\mathcal{G}) : W \mapsto L^2(\Omega)^3$ be the operator such that, for $G \in W$,
\begin{equation*}
A(\mathcal{G}) G := 
p(\mathcal{G}^x) G^x_{x} + q(\mathcal{G}^y) G^y_{y}. \\
\end{equation*}
Note that the operator $A(\G)$ is not uniformly elliptic.
We want to work with a uniformly elliptic operator in order to verify the Cordes condition, see \cite{smears,maugeri2000elliptic}.
As Equation \eqref{eq:ineq constraint} indicates that we are not interested in the values of $\nabla \varphi$ when $|\varphi_x| > 3$, $|\varphi_y| > 4$ and $|\varphi_y| \le 1$, we define the Lipschitz cut-offs,
\[ \bar{p}(\mathcal{G}^x) := \left\{ \begin{array}{cc} 4 & \text{if } |\mathcal{G}^x|^2 \geq 3 \\
p(\mathcal{G}^x) & \text{otherwise} \\
 \end{array} \right., \quad
 \bar{q}(\mathcal{G}^y) :=\left\{ \begin{array}{cc} 4 & \text{if } |\mathcal{G}^y|^2 \le 1 \\
 1 & \text{if } |\mathcal{G}^y|^2 \ge 4 \\
q(\mathcal{G}^y) & \text{otherwise} \\
 \end{array} \right. . \]
Therefore, $1 \le \bar{p}(\mathcal{G}^x) \le 4$ and $1 \le \bar{q}(\mathcal{G}^y) \le 4$ and $\bar{p}$ and $\bar{q}$ are $K$-Lipschitz functions, where $K>0$. 
Note that when \eqref{eq:ineq constraint} is verified, then $\bar{p}(\varphi_x) = p(\varphi_x)$ and $\bar{q}(\varphi_y) = q(\varphi_y)$.
The operator $\bar{A}(\mathcal{G})$ is defined for $G \in W$ as,
\begin{equation}
\label{eq:grad operator}
\bar{A}(\mathcal{G}) G := 
\bar{p}(\mathcal{G}^x) G^x_{x} + \bar{q}(\mathcal{G}^y) G^y_{y}. 
\end{equation}
The operator $\bar{A}(\G)$ is uniformly elliptic and thus verifies the Cordes condition.

\subsection{Gradient boundary conditions}

As the new variable $\mathcal{G}$ is assumed to be a gradient, it should verify a generalization of Schwarz's theorem, which states that for a distribution $f$, $f_{xy} = f_{yx}$ in the sense of distributions.
Therefore, we define $\mathbb{W} := \{G \in W | G^x_y = G^y_x \text{ a.e.~in } \Omega \}$, equipped with the usual $H^1(\Omega)$ norm, and the gradient Dirichlet boundary conditions should be in the trace of $\mathbb{W}$.
Let $\mathfrak{T}:H^1(\Omega)^{3 \times 2} \to H^\frac12(\partial \Omega)^{3 \times 2}$ be the trace operator, see \cite[Theorem B.52]{ern_guermond} for instance.
Let $\mathbb{H}^\frac12(\partial \Omega)^{3 \times 2} := \mathfrak{T} \mathbb{W}$ be a trace space,
equipped with the usual $H^\frac12(\partial \Omega)$ norm.

\begin{lemma}[Characterization of $\mathbb{H}^\frac12(\partial \Omega)$]
One has the following characterization:
\[ \mathbb{H}^\frac12(\partial \Omega)^{3 \times 2} = \left\{ \begin{pmatrix}
g_1^\mathsf{T} \\
g_2^\mathsf{T} \\
g_3^\mathsf{T} \\
\end{pmatrix} \, ; \, \forall i \in \{1,2,3\}, \,  g_i \in H^\frac12(\partial \Omega)^2 \text{ and } \int_{\partial \Omega} g_i \cdot \tau = 0 \right\} \]
where $\tau$ is a tangent vector to $\partial \Omega$.
\end{lemma}

\begin{proof}
Let us begin by showing the direct inclusion.
Let $G_D \in \mathbb{H}^\frac12(\partial \Omega)^{3 \times 2}$
As it is a trace space, there exists $\hat{G}_D \in H^1(\Omega)^{3 \times 2}$, $\hat{G}_D = G_D$ on $\partial \Omega$ and $\hat{G}^x_{D,y} = \hat{G}^y_{D,x}$ a.e. in $\Omega$.
Note that $0 = \hat{G}^y_{D,x} - \hat{G}^x_{D,y} = \mathrm{div}(\epsilon \hat{G}_D)$, where $\epsilon = \begin{pmatrix}
0 & 1 \\ -1 & 0 \\ \end{pmatrix}$.
Applying the divergence theorem, one has
\[ 0 = \int_\Omega \mathrm{div}(\epsilon \hat{G}_D) = \int_{\partial \Omega} \epsilon \hat{G}_D \cdot n = \int_{\partial \Omega}  G_D \cdot \tau, \] 
where $\tau = \epsilon^\mathsf{T} n$ is a tangent vector to $\partial \Omega$.
This proves the direct inclusion.

Let us now show the reverse inclusion.
For $i \in \{1,\ldots,3\}$, we consider $g_i \in H^\frac12(\partial \Omega)^2$ such that $\int_{\partial \Omega} g_i \cdot \tau = 0$.
We define $\tilde{g}_i := \epsilon g_i$.
One thus has $\int_{\partial \Omega} \tilde{g}_i \cdot n = 0$.
Let $i=1,2,3$.
We apply \cite[Lemma~2.2]{girault1979finite} and get the existence of $\tilde{u}_i \in H^1(\Omega)^2$ such that $\mathrm{div}(\tilde{u}_i) = 0$ a.e. in $\Omega$ and $\tilde{u}_i = \tilde{g}_i$ on $\partial \Omega$.
There also exists $C>0$, depending only on $\Omega$, such that
\begin{equation}
\label{eq:trace ineq}
\|\tilde{u}_i \|_{H^1(\Omega)} \le C \|\tilde{g}_i \|_{H^\frac12(\partial \Omega)}.
\end{equation}
We define $u_i := \epsilon \tilde{u}_i$.
Therefore, one has $u^x_{i,y} - u^y_{i,x} = 0$ a.e. in $\Omega$.
Defining $G := (u_i^\mathsf{T})_{i=1,2,3}$ and $G_D := (g_i^\mathsf{T})_{i=1,2,3}$, one has $G \in \mathbb H^1(\Omega)^{3 \times 2}$, $G_D \in H^\frac12(\partial \Omega)^{3 \times 2}$ and $G = G_D$ on $\partial \Omega$.
Using \eqref{eq:trace ineq} for each component, one gets
\[ \|G \|_{H^1(\Omega)} \le C \|G_D \|_{H^\frac12(\partial \Omega)}. \]
Noticing that $G^x_y - G^y_x = 0$ a.e. in $\Omega$, one can conclude that $G_D \in \mathfrak{T} \mathbb H^1(\Omega)^{3 \times 2} = \mathbb H^\frac12(\partial \Omega)^{3 \times 2}$.
\end{proof}

In the rest of this paper, we consider gradient Dirichlet boundary conditions $G_D$ that verify Hypothesis \ref{hyp}.
\begin{hyp}
\label{hyp}
Let $G_D \in \mathbb{H}^\frac12(\partial \Omega)^{3 \times 2}$ such that, a.e. on $\partial \Omega$,
\[ \left\{ \begin{aligned}
&0 < |G_D^x|^2 \le 3, \\
& |G_D^y|^2 = \frac4{4 - |G_D^x|^2} \le 4, \\
& G_D^x \cdot G_D^y = 0.
\end{aligned} \right.  \]
\end{hyp}
Note that, the subspace of $\mathbb{H}^\frac12(\partial \Omega)^{3 \times 2}$ of functions that verify Hypothesis \ref{hyp} is not empty.
Section \ref{sec:num examp} gives some examples of such functions.
We consider the convex subsets $W_D := \{G \in W \ |\ G = G_D \text{ on } \partial \Omega\}$, and $\mathbb{W}_D := \{G \in \mathbb{W} \ |\ G = G_D \text{ on } \partial \Omega\}$.

We finally want to find $\mathcal{G} \in \mathbb{W}_D$, such that
\begin{equation}
\label{eq:gradient eq}
\bar{A}(\mathcal{G}) \mathcal{G} = 0 \text{ a.e. in } \Omega.
\end{equation}
We propose to first study a linear equation where the coefficient of the operator $\bar{A}(\G)$ is frozen.
In a second time, we propose to use a fix point theorem, and the bounds obtained from studying the linear problem, to conclude the existence of solutions to \eqref{eq:gradient eq}.

\subsection{Linearized problem}
Let $\mathcal{G} \in W$.
The linearized equation we want to solve is thus, find $G \in \mathbb{W}_D$, 
\begin{equation}
\label{eq:first order inside}
\bar{A}(\mathcal{G}) G = 0 \text{ a.e. in } \Omega.
\end{equation}
To study the well-posedness of \eqref{eq:first order inside}, we write an appropriate mixed formulation, following \cite{gallistl2017variational}.
Let $R:= \{r \in L^2(\Omega)^3 | \int_\Omega r = 0 \}$, equipped with the usual $L^2$ norm.
Let $G \in W$.
We define the bilinear form, for all $\tilde{G} \in W_0$,
\begin{equation*}
a(\G;G,\tilde{G}) := \int_\Omega \bar{A}(\G) G \cdot \bar{A}(\G) \tilde{G},
\end{equation*}
and for all $\tilde{r} \in R$,
\begin{equation*}
b(G,\tilde{r}) = \int_\Omega (G^x_y - G^y_x) \cdot \tilde{r}.
\end{equation*}
The mixed formulation of \eqref{eq:first order inside} consists in seeking $(G,r) \in W_D \times R$,
\begin{subequations}
\label{eq:mixed continuous}
\begin{gather}
a(\G;G,\tilde{G}) + b(\tilde{G},r) = 0, \quad \forall \tilde{G} \in W_0 , \\
b(G,\tilde{r}) = 0, \quad \forall \tilde{r} \in R.
\end{gather}
\end{subequations}
Equation (\ref{eq:mixed continuous}b) will be shown in Lemma \ref{th:first order} to imply that $G$ verifies Clairault's theorem.

\begin{lemma}[Coercivity]
\label{th:coercivity}
Let $\gamma(\G)$ such that 
\[ \gamma(\G) := \frac{\bar{p}(\G^x) + \bar{q}(\G^y)}{\bar{p}(\G^x)^2 + \bar{q}(\G^y)^2}. \]
One has $\gamma_0 := \frac1{16} \le \gamma(\G) \le 4 =: \gamma_1$.
There exists $\varepsilon \in (0,1]$, independent of $\gamma(\G)$,
for all $G \in W_0$, such that $G^x_y = G^y_x$,
\[ \left(\frac{1 - \sqrt{1 - \varepsilon}}{\gamma_1} \right)^2 \Vert \nabla G \Vert_{L^2(\Omega)}^2 \le a(\G; G,G). \]
The coercivity constant is independent of $\G$.
\end{lemma}

\begin{proof}
Following \cite{gilbarg2015elliptic,smears}, as $\bar{A}(\G)$ verifies the Cordes condition, one can prove that there exists $\varepsilon \in (0,1]$, independent of $\G$, such that
\[ \Vert \gamma(\G) \bar{A}(\G) G \Vert_{L^2(\Omega)} \ge (1 - \sqrt{1 - \varepsilon}) \Vert \mathrm{div}(G) \Vert_{L^2(\Omega)} . \]
One also has, for all $G \in W_0$, 
\[ \Vert \nabla G \Vert^2_{L^2(\Omega)} \le \Vert G^x_y - G^y_x \Vert^2_{L^2(\Omega)} + \Vert \mathrm{div}(G) \Vert^2_{L^2(\Omega)}, \]
as shown in \cite[Theorem 2.3]{costabel1999maxwell}.
Therefore, for all $G \in W_0$ such that $G^x_y = G^y_x$, a.e.~in $\Omega$,
\[ \Vert \gamma(\G) \bar{A}(\G) G \Vert_{L^2(\Omega)} \ge (1 - \sqrt{1 - \varepsilon}) \Vert \nabla G \Vert_{L^2(\Omega)}. \]
Finally, one has
\[ (1 - \sqrt{1 - \varepsilon})^2 \Vert \nabla G \Vert_{L^2(\Omega)}^2 \le \gamma_1^2  \Vert \bar{A}(\G) G \Vert_{L^2(\Omega)}^2. \]

\end{proof}

\begin{lemma}
\label{th:first order}
Equation \eqref{eq:first order inside} admits a unique solution $G \in \mathbb{W}_D$.
There exists $C>0$, independent from $\G$,
\[ \Vert G \Vert_{H^1(\Omega)} \le C \Vert G_D \Vert_{H^\frac12 (\partial \Omega)}.  \]
\end{lemma}

\begin{proof}
As $G_D \in \mathbb{H}^{\frac12}(\partial \Omega)^{3 \times 2}$, there exists $\hat{G}_D \in H^1(\Omega)^{3 \times 2}$, $\hat{G}_D = G_D$ on $\partial \Omega$, and there exists $C>0$, independent from $\G$,
\[ \Vert \hat{G}_D \Vert_{H^1(\Omega)} \le C \Vert G_D \Vert_{H^\frac12(\partial \Omega)}. \]
Also, note that, as $G_D \in \mathbb{H}^\frac12(\partial \Omega)$, then $\hat{G}_D \in \mathbb{W}$.
We define
\[f := -\bar{A}(\mathcal{G}) \hat{G}_D \in L^2(\Omega)^3.\]
We seek $(\hat{G},r) \in W_0 \times R$,
\[ \left\{ \begin{aligned}
& a(\G;\hat{G},\tilde{G}) + b(\tilde{G},r) = \int_\Omega f \cdot \bar{A}(\G)\tilde{G}, \quad \forall \tilde{G} \in W_0, \\
& b(\hat{G},\tilde{r}) = -b(\hat{G}_D,\tilde{r}) = 0 , \quad \forall \tilde{r} \in R, \\
\end{aligned} \right. \]
as $\hat{G}^x_{D,y} = \hat{G}^y_{D,x}$ a.e. in $\Omega$.

We define $G := \hat{G} + \hat{G}_D$, which will be a solution of \eqref{eq:first order inside}.
We use the BNB lemma, see \cite[Theorem 2.6, p.~85]{ern_guermond}, in the context of saddle point problems \cite[Theorem 2.34, p.~100]{ern_guermond}.
As shown in \cite{gallistl2017variational}, there exists $\beta > 0$, independent of $\G$,
\[ \beta \le \inf_{\tilde{r} \in R\setminus \{0\}} \sup_{\hat{G} \in W_0\setminus \{0\}} \frac{b(\hat{G},\tilde{r})}{\Vert\tilde{r} \Vert_{R} \Vert \hat{G} \Vert_{W_0}}. \]
$a(\G)$ was proved to be coercive over $\mathbb{W}_{0}$ in Lemma \ref{th:coercivity}, thus fulfilling the two conditions of the BNB theorem.
Therefore, \eqref{eq:mixed continuous} admits a unique solution.
Applying the BNB theorem and a trace inequality gives
\[ \Vert G \Vert_{H^1(\Omega)} \le \Vert \hat{G}_D \Vert_{H^1(\Omega)} + C \Vert \bar{A}(\G) \hat{G}_D \Vert_{L^2(\Omega)} \le C' \Vert G_D \Vert_{H^\frac12(\partial \Omega)}, \]
where $C,C'>0$ are independent of $\G$.

Let us now show that a solution of \eqref{eq:mixed continuous} verifies \eqref{eq:first order inside}.
As $\hat{G} \in W_0$, and for all $\tilde{r} \in R$, $b(\hat{G},\tilde{r}) = 0$, following \cite{gallistl2017stable}, one has $\hat{G}^x_y = \hat{G}^y_x$ a.e. in $\Omega$.
Therefore, as $\hat{G}_D \in \mathbb{W}_D$, then $G \in \mathbb{W}_D$ and verifies $G^x_y = G^y_x$ a.e. in $\Omega$.
Testing (\ref{eq:mixed continuous}a) with $\bar{A}(\G) \hat{G}$, one has
\[ \int_\Omega \bar{A}(\G)\hat{G} \cdot \bar{A}(\G)\hat{G} = \int_\Omega f \cdot \bar{A}(\G)\hat{G}. \]
Therefore, $\hat{G}$ minimizes $\Vert \bar{A}(\G) \tilde{G} - f \Vert_{L^2(\Omega)}^2$, over $\tilde{G} \in \mathbb W_0$.
We note that Lemma \ref{th:coercivity} proved that $\bar{A}(\G)$ is coercive over $\mathbb W_0$ and therefore surjective from $\mathbb W_0$ onto $L^2(\Omega)$.
We can now conclude that $G$ solves \eqref{eq:first order inside}.
\end{proof}

\subsection{Nonlinear problem}
\label{sec:fixed point}
The existence of a solution to \eqref{eq:gradient eq} is proved through a fixed point method.
Let $T: H^1(\Omega)^{3 \times 2} \ni \G \mapsto G(\G) \in \mathbb{W}_D$ be the map that, given a $\G \in H^1(\Omega)^{3 \times 2}$, maps to the unique solution of \eqref{eq:first order inside}.
We define the following subset
\begin{equation}
B= \left\{ \tilde{G} \in \mathbb{W}_D ; \Vert \tilde{G} \Vert_{H^1(\Omega)} \le C \Vert G_D \Vert_{H^{\frac12}(\partial \Omega)}\right\},
\end{equation}
where $C>0$ is the constant from Lemma \ref{th:first order}.

\begin{proposition}
\label{th:fixed point}
The map $T$ admits a fixed point $\G \in B$ which verifies
\[ \bar{A}(\G)\G = 0 \text{ a.e. in } \Omega. \]
\end{proposition}

\begin{proof}
Let us show that $B$ is stable by $T$: $TB \subset B$.
Let $\G \in B$.
Letting $G:= T\G$, one thus has $G \in B$, as a consequence of Lemma \ref{th:first order}.
Also, $B \subset L^2(\Omega)^{3 \times 2}$ with a compact Sobolev embedding and $L^2(\Omega)^{3 \times 2}$ equipped with its usual norm is a Banach space.
Thus, $B$ is a closed convex subset of the Banach space $L^2(\Omega)^{3 \times 2}$ and $TB$ is precompact in $L^2(\Omega)^{3 \times 2}$.

We prove that $T$ is continuous over $B$.
Let $\left(\G_n\right)_n$ be a sequence of $B$ such that $\G_n \mathop{\longrightarrow} \limits_{n \to +\infty} \G \in B$, for $\Vert \cdot \Vert_{H^1(\Omega)}$.
Let $G := T \G$ and $G_n := T \G_n$, for $n \in \mathbb{N}$.
We want to prove that $G_n \mathop{\longrightarrow} \limits_{n \to +\infty} G$ for $\Vert \cdot \Vert_{H^1(\Omega)}$.
As for all $n \in \mathbb{N}$, $G_n \in B$, $\left( G_n \right)_n$ is bounded in $H^1(\Omega)^{3 \times 2}$.
Therefore, there exists $\hat{G} \in H^1(\Omega)^{3 \times 2}$, up to a subsequence,  $G_n \mathop{\rightharpoonup} \limits_{n \to +\infty} \hat{G}$ weakly in $H^1(\Omega)$.
Due to the Rellich--Kondrachov theorem, $G_n \mathop{\longrightarrow} \limits_{n \to +\infty} \hat{G}$ strongly in $L^2(\Omega)$.

Let us show that $\hat{G}$ is a solution of \eqref{eq:first order inside}.
By definition, one has
\begin{equation}
\label{eq:aux2}
\begin{aligned}
\Vert \bar{A}(\G)G_n \Vert_{L^2(\Omega)} &= \Vert \bar{A}(\G)G_n - \bar{A}(\G_n)G_n \Vert_{L^2(\Omega)} \\
& \le K \Vert \G - \G_n \Vert_{H^1(\Omega)}  \Vert G_n \Vert_{H^1(\Omega)} \\
& \le C K \Vert G_D \Vert_{H^\frac12(\partial \Omega)} \Vert \G - \G_n \Vert_{H^1(\Omega)} \mathop{\longrightarrow}_{n \to \infty} 0,
\end{aligned}
\end{equation}
since $\bar{A}$ is $K$-Lipschitz.
The left-hand side converges to $\Vert \bar{A}(\G) \hat{G} \Vert_{L^2(\Omega)}$ when $n \to +\infty$.
As for all $n \in \mathbb{N}$, $G_n \in \mathbb{W}_D$, then  $\hat{G} \in \mathbb{W}_D$ and $\hat{G} = G_D$ on $\partial \Omega$.
Therefore, $\hat{G}$ solves \eqref{eq:first order inside}.
As proved in Lemma \ref{th:first order}, the solution of \eqref{eq:first order inside} is unique in $\mathbb{W}_D$. 
Thus, the full sequence $\left( G_n \right)_n$ converges strongly towards $\hat{G} = G = T \G$ in $L^2(\Omega)$.

To finish showing the continuity of $T$ over $H^1(\Omega)$, we notice that by definition, $\bar{A}(\G) (G_n - G) = \bar{A}(\G)G_n$.
Also note that $G - G_n \in W_0$ and $G^x_y - G^x_{n,y} = G^y_x - G^y_{n,x}$ a.e. in $\Omega$.
Using the coercivity proved in Lemma \ref{th:coercivity}, one has
\[ C \|\nabla G - \nabla G_n \|_{L^2(\Omega)}^2 \le a(\G;G - G_n, G - G_n) = \| \bar{A}(\G)(G - G_n) \|^2_{L^2(\Omega)} = \| \bar{A}(\G)G_n \|^2_{L^2(\Omega)} \to 0, \]
as $n \to +\infty$, by using \eqref{eq:aux2}.
Therefore, $G_n \to G$ strongly in $H^1(\Omega)$.
We conclude with the Schauder fixed point Theorem, see Corollary 11.2 of \cite{gilbarg2015elliptic}, that $T$ admits a fixed point.
\end{proof}

\begin{remark}
Proposition \ref{th:fixed point} shows that \eqref{eq:gradient eq} has solutions but does not address uniqueness.
We believe uniqueness cannot be proved with conventional tools.
Indeed, uniqueness is generally obtained by considering a boundary condition $G_D$ whose $H^\frac12(\partial \Omega)$ norm is small enough to make the fix point map $T$ contracting.
However, Hypothesis \ref{hyp} considers $|G_D^y| > 1$ a.e. on $\partial \Omega$, which entails
\[ \Vert G_D \Vert_{H^\frac12(\partial \Omega)} \ge 1. \]
This prevents one from considering small enough boundary conditions.
\end{remark}

The following allows one to recover solutions of \eqref{eq:nonlinear operator eq} from solutions of \eqref{eq:gradient eq}.
\begin{coro}
\label{th:phi}
Let $\G \in \mathbb{W}_D$ be a solution of \eqref{eq:gradient eq}, and $\mathcal{V} := \{\varphi \in V \ | \ \int_\Omega \varphi = 0\}$.
There exists a unique $\varphi \in \mathcal{V}$,
\[ \nabla \varphi = \mathcal{G} \text{ a.e. in } \Omega,\]
which verifies
\begin{equation}
\label{eq:cut-off main eq phi}
\bar{p}(\varphi_x) \varphi_{xx} + \bar{q}(\varphi_y) \varphi_{yy} = 0 \text{ in } \Omega.
\end{equation}
\end{coro}

\begin{proof}
Let $\mathcal{X} = \{\psi \in H^1(\Omega)^3 \ | \ \int_\Omega \psi = 0\}$.
As $\G \in \mathbb{W}$, it verifies $\G^x_y = \G^y_x$ a.e. in $\Omega$.
As $\Omega$ is a bounded, simply-connected, domain with a Lipschitz boundary, we use the Helmoltz decomposition, see \cite{helmoltz} for instance, and there exists $\varphi \in \mathcal{X}$ such that $\nabla \varphi = \G$ in $L^2(\Omega)$.
As $\G \in W$, $\varphi \in H^2(\Omega)$.
Also, one has
\[ \bar{A}(\G) \G = 0, \]
and therefore, one has in $L^2(\Omega)$,
\[ \bar{p}(\varphi_x) \varphi_{xx} + \bar{q}(\varphi_y) \varphi_{yy} = 0.\]
\end{proof}

\subsection{Constrained problem}
Proposition \ref{th:phi} proved that there exists $\varphi \in \mathcal{V}$ which solves \eqref{eq:cut-off main eq phi}.
This section addresses the verification of the constraints \eqref{eq:microstructure} - \eqref{eq:ineq constraint}.
The argument used in this section requires some more regularity on $\G$, solution of \eqref{eq:nonlinear operator eq}.
In order to obtain $\G \in \mathcal{C}^0(\bar{\Omega})$, we assume the following.

\begin{hyp}
\label{hyp2}
The Dirichlet boundary condition has the extra regularity $G_D \in W^{1-\frac1{s},s}(\partial \Omega)$ for $2 < s < 4$.
\end{hyp}
This hypothesis will be used in the rest of this section.
Using Hypothesis \ref{hyp2}, our objective is to prove that $\G \in \mathbb{W}^{1,r}(\Omega) := \{ G \in W^{1,r}(\Omega) ; G^x_y = G^y_x \text{ a.e. in } \Omega \}$, for a certain $2 < r < s$.

\begin{lemma}
\label{th:extra regularity}
Let $\mathcal G \in \mathbb W_D$ be a solution of \eqref{eq:gradient eq} where $G_D$ verifies Hypothesis \ref{hyp2}.
There exists $2 < r < s < 4$, $\mathcal G \in W^{1,r}(\Omega)$.
Therefore, $\mathcal G \in \mathcal{C}^0(\bar{\Omega})$.
\end{lemma}

\begin{proof}
We are going to write a mixed formulation different from \eqref{eq:mixed continuous}.
First, we lift the boundary conditions: there exists $\hat{G}_D \in \mathbb W^{1,s}(\Omega)^{3 \times 2}$ such that $\hat{G}_D = G_D$ on $\partial \Omega$.
We define the associated right-hand side
\[f := -\bar{A}(\mathcal{G}) \hat{G}_D \in L^s(\Omega)^3.\]
Let $2 < r < s$.
For $G, \tilde{G} \in W^{1,r}_0(\Omega)^{3 \times 2}$, we define the bilinear form,
\begin{equation*}
\mathfrak{a}(G,\tilde{G}) := \int_\Omega \bar{A}(\G) G \cdot\mathrm{div}(\tilde{G}).
\end{equation*}
Let $T := \{t \in L^{r'}(\Omega)^3 ; \int_\Omega t = 0\}$, where $\frac1{r} + \frac1{r'} = 1$.
For $G \in W^{1,r}_0(\Omega)^{3 \times 2}$ and $\tilde{t} \in T$, we define
\begin{equation*}
\mathfrak{b}(G,\tilde{t}) = \int_\Omega \mathrm{div}(\epsilon G) \cdot \tilde{t}.
\end{equation*}
The mixed formulation consists in seeking $(G,t) \in W^{1,r}_0(\Omega)^{3 \times 2} \times T$,
\begin{subequations}
\label{eq:mixed2}
\begin{gather}
\mathfrak a(G,\tilde{G}) + \mathfrak b(\tilde{G},t) = 0, \quad \forall \tilde{G} \in W^{1,r}_0(\Omega)^{3 \times 2} , \\
\mathfrak b(G,\tilde{t}) = 0, \quad \forall \tilde{t} \in T.
\end{gather}
\end{subequations}

Our first objective is to prove that $\bar{A}(\mathcal G)$ is bijective from $\mathbb W^{1,r}_0(\Omega)$, the homogeneous space associated to $\mathbb W^{1,r}(\Omega)$, onto $L^r(\Omega)$.
We follow a proof similar to the one of Lemma \ref{th:coercivity}.
We define $\gamma \equiv \gamma(\mathcal G)$ as in Lemma \ref{th:coercivity}.
As $\bar{A}(\G)$ is uniformly elliptic, it verifies the Cordes condition.
Thus, there exists $\varepsilon \in (0,1]$, independent of $\G$, such that
\[ \frac{\bar{p}(\G^x)^2 + \bar{q}(\G^y)^2}{(\bar{p}(\G^x) + \bar{q}(\G^y))^2} \le \frac1{1 + \varepsilon}. \]
Using the bounds $1 \le p(\G^x), q(\G^y) \le 4$, one can show that any $\varepsilon < \frac1{32}$ works.

We now show the first condition of the BNB theorem.
Let $G \in \mathbb W^{1,s}_0(\Omega)$.
Following \cite{gallistl2017variational}, one has a.e. in $\Omega$,
\[  \gamma |\bar{A}(\G) G| \ge |\mathrm{div}(G)| - \sqrt{1 - \varepsilon} |\nabla G|. \]
In \cite{smears}, the Miranda--Talenti estimate $\|\Delta v \|_{L^2(\Omega)} = \| \nabla^2 v \|_{L^2(\Omega)}$, for $v \in H^2(\Omega)^3 \cap H^1_0(\Omega)^3$, where $\nabla^2$ denotes the Hessian, is advocated to obtain coercivity.
In \cite{gallistl2017stable,gallistl2019numerical}, a similar $L^2$ based result is used to obtain coercivity.
Here, we instead advocate that one can use the Calderon--Zygmund inequality, see \cite[Corollary~9.10]{gilbarg2015elliptic}.
We now use a generalization of the Helmoltz decomposition to $L^p(\Omega)$ for $p > 2$, (see e.g. \cite[Theorem~1.2]{MR1883390} or \cite{MR1867691}).
As $G^x_y = G^y_x$ a.e. in $\Omega$, we deduce that there exists $v \in W^{2,s}_0(\Omega)^3$ such that $\nabla v = G$ in $L^s(\Omega)$.
Using the Calderon--Zygmund inequality, there exists $C > 0$, depending on $\Omega$ and $s$, such that
\[ \| \nabla^2 v \|_{L^s(\Omega)} \le C \| \Delta v \|_{L^s(\Omega)}.  \] 
One thus has
\[ \| \nabla G \|_{L^s(\Omega)} \le C \| \mathrm{div}(G) \|_{L^s(\Omega)}.  \] 
If the constant in the previous inequality is such that $C < 1$, then for any $2< r \le s$,
\[ \| \gamma \bar{A}(\G) G \|_{L^r(\Omega)} \ge \left(\frac1C - \sqrt{1 - \varepsilon} \right) \| \nabla G \|_{L^r(\Omega)},  \]
with a nonnegative constant.
But if $C > 1$, then we use the Riesz--Thorin interpolation theorem for $2 < r < s$, see \cite[Theorem~A.27]{ern2021finite}. Thus, for any $\alpha \in (0,1)$, one has
\[ \| \nabla G \|_{L^r(\Omega)} \le C^\alpha \| \mathrm{div}(G) \|_{L^r(\Omega)},  \]
where the value of $r > 2$ depends on $\alpha$.
We can choose $\alpha$ close enough to zero such that $C^{-\alpha} > \sqrt{1 - \varepsilon}$.
Therefore,
\[ \| \gamma \bar{A}(\G) G \|_{L^r(\Omega)} \ge \left(C^{-\alpha} - \sqrt{1 - \varepsilon} \right) \| \nabla G \|_{L^r(\Omega)},  \]
with a nonnegative constant. 
Similarly to the end of the proof of Lemma \ref{th:coercivity},
the upper bound on $\gamma$ provides the inf-sup stability of $\bar{A}(\G)$ over $\mathbb W_0^{1,r}(\Omega)^{3 \times 2}$.

Let us now show the second condition of the BNB theorem.
Let $\tilde{G} \in \mathbb W^{1,r}_0(\Omega)$ such that, $\mathfrak{a}(G,\tilde{G}) = 0$, for all $G \in \mathbb W^{1,r}_0(\Omega)$.
We want to show that $\tilde{G} = 0$.
Using the coercivity result from \cite{smears,gallistl2017variational}, one has for $\tilde{G} \in \mathbb{H}^1_0(\Omega)$,
\[ 0 = \mathfrak a(\tilde{G}, \tilde{G}) \ge C' \| \nabla \tilde{G} \|_{L^2(\Omega)}^2, \] 
where $C' > 0$ is a constant.
Thus $\tilde{G} = 0$.

Regarding the inf-sup stability for $\mathfrak b$, we refer to \cite[Remark~53.10]{MR4269305}.
One can conclude with the Babuska--Brezzi theorem, see \cite[Theorem~49.13]{MR4269305}, that there exists a unique solution to \eqref{eq:mixed2}.
As $\G$ is a fixed point, then it is the unique solution and benefits from the extra regularity.
A classical Sobolev injection finishes the proof.
\end{proof}

\begin{lemma}
\label{th:make solution}
Let $\varphi \in \mathcal{V}$ be a solution of Equation \eqref{eq:cut-off main eq phi} and assume Hypothesis \ref{hyp2} applies,
There exists a neighborhood $\Omega'$ of $\partial \Omega$ where $1 < |\varphi_y|^2 < 4$ and $|\varphi_x|^2 < 3$ a.e. in $\Omega'$.
We define in $\Omega'$,
\[ \left\{ \begin{aligned} 
u &:= \varphi_x \cdot \varphi_y, \\
 v &:= \log\left((1-\frac14|\varphi_x|^2)|\varphi_y|^2\right).
 \end{aligned} \right. \]
Then $(u,v) \in \mathcal C^0(\overline{\Omega'})^2 \cap H^1(\Omega')^2$,
and it verifies
\begin{equation}
\left\{
\begin{aligned}
&\mathfrak{p} u_x + 2 v_y = 0, \\
&\mathfrak{q} u_y - 2 v_x = 0, \\
\end{aligned}
\right.
\label{eq:ODE}
\end{equation}
where $\mathfrak{p} \equiv p(\varphi_x)$ and $\mathfrak{q} \equiv q(\varphi_y)$.
\end{lemma}

\begin{proof}
Let us note that, as a consequence of Lemma \ref{th:extra regularity}, $\nabla \G \in \mathcal C^0(\bar{\Omega})$.
We reason by contradiction.
Let us assume that there exists an open ball $B$ touching $\partial \Omega$ such that $|\varphi_y | \le 1$ a.e. in $B \cap \Omega$.
By \cite[Theorem~3.87]{ambrosio2000functions}, for $\mathcal H^1$-a.e.~$x \in \partial \Omega$, where $\mathcal H^1$ is the one-dimensional Hausdorff measure, there exists $\varphi_y^\Omega(x) \in \mathbb{R}^3$, such that
\[ \lim_{\rho \to 0} \rho^{-2} \int_{\Omega \cap B_\rho(x)} |\varphi_y(z) - \varphi_y^\Omega(x)| dz = 0, \]
where $B_\rho(x) \subset \mathbb{R}^2$ is the ball of radius $\rho > 0$ centered at $x \in \Omega$. 
By hypothesis, for a.e. $x \in B$, $|\varphi_y(x)| \le 1$.
Therefore, as $\varphi_y$ is continuous, for a.e. $x \in \partial \Omega \cap \partial (B \cap \Omega)$, $|\varphi_y^\Omega(x)| \le 1$.
This contradicts Hypothesis \ref{hyp}.
Thus, there exists an open domain $\Omega'_y \subset \Omega$ such that $1 < |\varphi_y|$ in $\Omega'_y$ and $\partial \Omega'_y \supset \partial \Omega$.
One can repeat the reasoning for the upper bound and for $|\varphi_x|^2 < 3$.
The final domain $\Omega'$ is then the intersection of these three domains.
Note that $\Omega' \neq \emptyset$ as the three domains are neighborhoods of $\partial \Omega$.

Using the definition of $\Omega'$, one has $v \in L^\infty(\Omega')$.
Computing the derivatives in $\mathcal{D}'(\Omega')$ (in the sense of distributions), one has
\[ \left\{ \begin{aligned} 
&u_x = \varphi_{xx} \cdot \varphi_y + \varphi_x \cdot \varphi_{xy}, \\
&u_y = \varphi_{xy} \cdot \varphi_y + \varphi_x \cdot \varphi_{yy},
\end{aligned} \right. \]
and
\[ \left\{ \begin{aligned} 
&v_x = -\frac{\mathfrak{p}}2 \varphi_{xx} \cdot \varphi_x + \frac{\mathfrak{q}}2\varphi_y \cdot \varphi_{yx}, \\
&v_y = -\frac{\mathfrak{p}}2 \varphi_{xy} \cdot \varphi_x + \frac{\mathfrak{q}}2 \varphi_y \cdot \varphi_{yy}.
\end{aligned} \right. \]
As $\varphi \in H^2(\Omega)$, as shown in Proposition \ref{th:phi}, and as $1 \le \mathfrak{p}, \mathfrak{q} \le 4$, one has $\nabla u \in L^2(\Omega')^2$ and $\nabla v \in L^2(\Omega')^2$.
Therefore, $(u,v) \in H^1(\Omega')^2$.
 
Let us now show that \eqref{eq:ODE} is verified.
Equation \eqref{eq:cut-off main eq phi} is projected onto $\varphi_x$ and $\varphi_y$.
One thus has in $L^2(\Omega')$,
\[ \left\{ \begin{aligned} 
&\mathfrak{p} \varphi_{xx}\cdot \varphi_x + \mathfrak{q} \varphi_{yy} \cdot \varphi_x = 0,\\
&\mathfrak{p} \varphi_{xx} \cdot \varphi_y + \mathfrak{q} \varphi_{yy} \cdot \varphi_y = 0.
\end{aligned} \right. \]
Finally, one has in $L^2(\Omega')$,
\[ \left\{ \begin{aligned}
& \mathfrak{p} u_x = - 2 v_y, \\
& \mathfrak{q} u_y = 2 v_x.
\end{aligned} \right. \]
\end{proof}

Lemma \ref{th:make solution} ensures that \eqref{eq:ineq constraint} is verified in $\Omega' \subset \Omega$.
Note that two cases are possible with $\Omega'$: whether $\Omega' = \Omega$, or $\Omega' \subsetneq \Omega$ has one connected component.

\begin{proposition}
\label{th:two eq}
Let us assume that $\Omega' = \Omega$.
The unique solution in $H^1_0(\Omega)^2$ of \eqref{eq:ODE}  is the pair $(0,0)$.
\end{proposition}

\begin{proof}
Let us take the derivatives, in the sense of $\mathcal{D}'(\Omega)$, in $x$ of the first line of \eqref{eq:ODE} and in $y$ of the second line of \eqref{eq:ODE},
\[ \left\{ \begin{aligned}
& - (\mathfrak{p} u_x)_x = 2 v_{yx}, \\
& (\mathfrak{q} u_y)_y = 2v_{xy}.
\end{aligned} \right. \] 
Using the equality of cross derivatives, one has
\[ \mathrm{div}(\mathfrak{A}(z) \nabla u) = 0, \]
where $z=(x,y)$ and $\mathfrak{A}(z) = \begin{pmatrix}
\mathfrak{p}(z) & 0 \\ 0 & \mathfrak{q}(z)
\end{pmatrix}$.
Using \cite[Theorem~8.3]{gilbarg2015elliptic}, there exists a unique $u \in H^1_0(\Omega)$ solution of that linear equation.
We can now apply \cite[Corollary~8.2]{gilbarg2015elliptic} to show that $u = 0$ a.e. in $\Omega$.
Therefore, using \eqref{eq:ODE}, as $v=0$ on $\partial \Omega$, one also has $v = 0$ a.e. in $\Omega$.
\end{proof}

\begin{remark}
It does not seem possible to remove the assumption that $\Omega' = \Omega$ from Proposition \ref{th:make solution} as one does not know if $(u,v)=(0,0)$ on $\partial \Omega' \setminus \partial \Omega$ when $\Omega' \subsetneq \Omega$.
\end{remark}

The consequence of Lemma \ref{th:make solution} and Proposition \ref{th:two eq} consists in stating that as $(u,v) = 0$ on $\partial \Omega$ due to Hypothesis \ref{hyp}, if $\Omega' = \Omega$, then $(u,v) = 0$ a.e.~in $\Omega$.
The main result of this section is stated thereafter.

\begin{theorem}
\label{th:general}
Assuming the Dirichlet boundary condition $G_D$ satisfies Hypotheses \ref{hyp} and \ref{hyp2}, there exists $\G \in \mathbb{W}_D$ solution of \eqref{eq:gradient eq}.
There also exists $\varphi \in \mathcal{V}$ such that $\nabla \varphi = \G$ in $\Omega$ and $\Omega' \subset \Omega$, open set in $\mathbb{R}^2$, such that $\varphi$ is a solution of \eqref{eq:min surface eq} and \eqref{eq:ineq constraint} in $\Omega'$.
If also $\Omega' = \Omega$, then $\varphi$ is a solution of \eqref{eq:strong form equations} a.e. in $\Omega$.
\end{theorem}

\begin{proof}
Using Proposition \ref{th:fixed point}, there exists $\G \in \mathbb{W}_D$ solution of \eqref{eq:gradient eq}.
Using Proposition \ref{th:phi}, there exists $\varphi \in \mathcal{V}$, $\nabla \varphi = \G$.
Using Lemma \ref{th:make solution}, there exists $\Omega' \subset \Omega$ where \eqref{eq:ineq constraint} is verified.
Due to the imposed boundary conditions, one has $(u,v)=(0,0)$ on $\partial \Omega$ in the notation of Lemma \ref{th:make solution}.
If $\Omega' = \Omega$, then using Proposition \ref{th:two eq}, one has $(u,v) = (0,0)$ a.e.~in $\Omega$.
Therefore,  a.e.~in $\Omega$,
\[ \left\{ \begin{aligned}
&\varphi_x \cdot \varphi_y = 0, \\
&\left(1 - \frac14|\varphi_x|^2\right)|\varphi_y|^2 = 1.
\end{aligned} \right.  \]
Thus, \eqref{eq:microstructure} and \eqref{eq:local basis} are verified in $\Omega$.
\end{proof}

Given Hypothesis \ref{hyp}, the only challenging part of \eqref{eq:ineq constraint} is $1 < |\varphi_y|^2$.
Indeed, with arguments along the lines of maximum principles, one expects the upper bounds to be verified.
Therefore, one expects that it is the constraint $1 < |\varphi_y|^2$ that will be determining between $\Omega' = \Omega$ and $\Omega' \subsetneq \Omega$.
Section \ref{sec:annulus} below presents a case in which $\Omega' \neq \Omega$.
The physical interpretation is that on the curves where $1 = |\varphi_y|^2$, the Miura fold is fully folded.
Therefore, the solution computed in the domain where $1 > |\varphi_y|^2$ is not physical.
The techniques used in this paper cannot predict a priori if $\Omega' = \Omega$.
However, assuming $\Omega' = \Omega$ for a given $G_D$, verifying Hypothesis \ref{hyp}, one can expects that a small modification of the Dirichlet boundary conditions, assuming it still verifies Hypothesis \ref{hyp}, will still entail $\Omega' = \Omega$.
Section \ref{sec:deformed} presents such a case.

\section{Numerical method}
In \cite{marazzato}, an $H^2$-conformal finite element method was used to approximate the solutions of \eqref{eq:min surface eq}.
In this paper, a simpler approximation, based on a stabilized formulation, and using $\mathbb{P}^1$--Lagrange finite elements is proposed for an improved efficiency.
Indeed, the new method requires much less degrees of freedom with respects to \cite{marazzato}.
A Newton method is then used to solve the discrete system.

\subsection{Discrete Setting}
Let $\left(\mathcal{T}_h \right)_h$ be a family of quasi-uniform and shape regular triangulations \cite{ern_guermond}, perfectly fitting $\Omega$.
For a cell $c \in \mathcal{T}_h$, let $h_c := \mathrm{diam}(c)$ be the diameter of $c$.
Then, we define $h := \max_{c \in \mathcal{T}_h} h_c$ as the mesh parameter for a given triangulation $\mathcal{T}_h$.
Let $W_h := \mathbb{P}^1(\mathcal{T}_h)^{3 \times 2}$, where the corresponding interpolator is written as $\mathcal{I}_h$.
We also define the following solution space,
\[ W_{hD} := \{G_h \in W_{h}; G_h =  \mathcal{I}_h G_D \text{ on } \partial \Omega  \}, \]
and its associated homogeneous space
\[ W_{h0} := \{G_h \in W_{h}; G_h =  0 \text{ on } \partial \Omega  \}. \]

\subsection{Discrete problem}
We follow \cite[Remark~3.5]{gallistl2019numerical} to introduce a stabilized approximation of \eqref{eq:gradient eq}.
It will give rise to a simpler discrete system than discretizing a mixed system.
In order to approximate a solution of \eqref{eq:gradient eq}, we will first ``freeze the coefficients" and obtain estimates on the linear problem and then apply a fix point theorem.
Let $\G_h \in W_{hD}$.
We search for $G_h \in W_{hD}$ such that
\begin{equation}
\label{eq:linearized discrete inside}
a(\G_h;G_h,\tilde{G}_h) + \mathsf{p}(G_h, \tilde{G}_h) = 0, \quad \forall \tilde{G}_h \in W_{h0},
\end{equation}
where we define the penalty bilinear form $\mathsf{p}$ such that for $G_h, \tilde{G}_h \in W_h$,
\[ \mathsf{p}(G_h, \tilde{G}_h) := \eta \int_\Omega (G_{h,y}^x - G_{h,x}^y) \cdot (\tilde{G}_{h,y}^x - \tilde{G}_{h,x}^y), \]
where $\eta > 0$ is a penalty parameter.
The proof of Lemma \ref{th:unique linearized} details the possible values of $\eta > 0$.

\begin{lemma}
\label{th:unique linearized}
There exists a unique solution $G_h \in W_{hD}$ of \eqref{eq:linearized discrete inside}.
This solution verifies
\begin{equation}
\label{eq:discrete bound}
\Vert G_h \Vert_{H^1(\Omega)} \le C \Vert G_D \Vert_{H^\frac12(\partial \Omega)},
\end{equation}
where $C>0$ is a constant independent of $\G_h$.
\end{lemma}

\begin{proof}
There exists $\hat{G}_D \in H^2(\Omega)^{3 \times 2} \subset \mathcal{C}^0(\Omega)^{3 \times 2}$ such that $\hat{G}_D = G_D$ on $\partial \Omega$.
We define the linear form for $\tilde{G}_h \in W_{h0}$,
\[ L(\G_h; \tilde{G}_h) := -\int_\Omega \bar{A}(\G_h)\mathcal{I}_h \hat{G}_D \cdot \bar{A}(\G_h) \tilde{G}_h, \]
which is bounded on $W_{h0}$, independently of $\G_h$.
Therefore, we search for $\hat{G}_h \in W_{h0}$,
\[ a(\G_h;\hat{G}_h,\tilde{G}_h) + \mathsf{p}(\hat{G}_h,\tilde{G}_h) = L(\G_h; \tilde{G}_h), \quad \forall \tilde{G}_h \in W_{h0}, \]
where $G_h = \hat{G}_h + \mathcal{I}_h \hat{G}_D$.
We want to prove the coercivity of $a(\G_h) + \mathsf{p}$.
Using Lemma \ref{th:coercivity}, applied to $\hat{G}_h$, one has
\begin{align*}
\Vert \mathrm{div}(\hat{G}_h) \Vert_{L^2(\Omega)} &\le (1 - \sqrt{1 - \varepsilon}) \Vert \mathrm{div}(\hat{G}_h) \Vert_{L^2(\Omega)} + \sqrt{1 - \varepsilon} \Vert \mathrm{div}(\hat{G}_h) \Vert_{L^2(\Omega)}, \\
&\le  \gamma_1 \Vert \bar{A}(\G_h)\hat{G}_h \Vert_{L^2(\Omega)} + \sqrt{1 - \varepsilon} \Vert \mathrm{div}(\hat{G}_h) \Vert_{L^2(\Omega)}, \\
& \le \gamma_1 \Vert \bar{A}(\G_h)\hat{G}_h \Vert_{L^2(\Omega)} + \sqrt{1 - \varepsilon} (\Vert \mathrm{div}(\hat{G}_h) \Vert_{L^2(\Omega)} + \Vert \hat{G}^y_{h,x} - \hat{G}^x_{h,y} \Vert_{L^2(\Omega)}).
\end{align*}
Thus,
\[ \frac{1 - \sqrt{1 - \varepsilon}}{\gamma_1} \Vert \mathrm{div}(\hat{G}_h) \Vert_{L^2(\Omega)} \le \Vert \bar{A}(\G_h)\hat{G}_h \Vert_{L^2(\Omega)} +  \frac{\sqrt{1 - \varepsilon}}{\gamma_1} \Vert \hat{G}^y_{h,x} - \hat{G}^x_{h,y} \Vert_{L^2(\Omega)}. \]
Therefore, using a Young inequality,
\[ \left(\frac{1 - \sqrt{1 - \varepsilon}}{\gamma_1} \right)^2 \Vert \mathrm{div}(\hat{G}_h) \Vert_{L^2(\Omega)}^2 \le (1 + \lambda) \Vert \bar{A}(\G_h)\hat{G}_h \Vert_{L^2(\Omega)}^2 + \left(1 + \frac1\lambda \right) \frac{1 - \varepsilon}{\gamma_1^2} \Vert \hat{G}^y_{h,x} - \hat{G}^x_{h,y} \Vert_{L^2(\Omega)}^2, \]
where $\lambda > 0$ is a parameter.
Finally,
\begin{multline*}
\left(\frac{1 - \sqrt{1 - \varepsilon}}{\gamma_1} \right)^2 \Vert \nabla \hat{G}_h \Vert_{L^2(\Omega)}^2 \le (1 + \lambda) \Vert \bar{A}(\G_h)\hat{G}_h \Vert_{L^2(\Omega)}^2 \\ + \left[ \left(\frac{1 - \sqrt{1 - \varepsilon}}{\gamma_1} \right)^2 + \left(1 + \frac1\lambda \right) \frac{1 - \varepsilon}{\gamma_1^2} \right] \Vert \hat{G}^y_{h,x} - \hat{G}^x_{h,y} \Vert_{L^2(\Omega)}^2.
\end{multline*}
Thus there exists $C>0$, independent of $\G_h$,
\[ C \Vert \hat{G}_h \Vert_{H^1(\Omega)}^2 \le a(\G_h;\hat{G}_h,\hat{G}_h) + \mathsf{p}(\hat{G}_h,\hat{G}_h), \]
by choosing adequately a value of $\eta > 0$, depending on $\lambda > 0$.
As a consequence of the Lax--Milgram lemma, one has
\[\Vert G_h \Vert_{H^1(\Omega)} \le \Vert \mathcal{I}_h \hat{G}_D \Vert_{H^1(\Omega)} + \Vert \hat{G}_h \Vert_{H^1(\Omega)}
\le C' \Vert \mathcal{I}_h \hat{G}_D \Vert_{H^1(\Omega)}
\le C'' \Vert G_D \Vert_{H^\frac12(\partial \Omega)}, \]
where $C',C''>0$ are independent from $\G_h$.
\end{proof}

\begin{proposition}
\label{th:nonlinear discrete eq}
There exists a solution $\G_h \in W_{hD}$ of
\begin{equation}
\label{eq:discrete eq}
a(\G_h; \G_h, \tilde{G}_h) + \mathsf{p}(\G_h, \tilde{G}_h) = 0, \quad \forall \tilde{G}_h \in W_{h0}.
\end{equation}

\end{proposition}

\begin{proof}
Let us define
\[ B_h := \left\{G_h \in W_{hD} ; \Vert G_h \Vert_{H^1(\Omega)} \le C \Vert G_D \Vert_{H^\frac12(\partial \Omega)} \right\}, \]
where $C>0$ is the constant from Lemma \ref{th:unique linearized}.
Let $T_h : W_{hD} \to W_{hD}$ be such that $T_h : \G_h \mapsto G_h$, where $G_h$ is the unique solution of \eqref{eq:linearized discrete inside}.
Equation \eqref{eq:discrete bound} shows that $T_h B_h \subset B_h$.

We now prove that $T_h$ is continuous over $B_h$.
Let $\left(\G_n\right)_n$ be a sequence of $B_h$ such that $\G_n \mathop{\longrightarrow} \limits_{n \to +\infty} \G_\infty \in B_h$ strongly in $H^1(\Omega)$.
Let $G_\infty := T_h \G_\infty$ and $G_n := T_h \G_n$, for $n \in \mathbb{N}$.
We want to prove that $G_n \mathop{\longrightarrow} \limits_{n \to +\infty} G_\infty$ strongly in $H^1(\Omega)$.
As for all $n \in \mathbb{N}$, $G_n \in B_h$, $\left( G_n \right)_n$ is bounded in $H^1(\Omega)^{3 \times 2}$.
Therefore, there exists $\hat{G} \in H^1(\Omega)^{3 \times 2}$, up to a subsequence,  $G_n \mathop{\rightharpoonup} \limits_{n \to +\infty} \hat{G}$ weakly in $H^1(\Omega)$.
Due to the Rellich--Kondrachov theorem, $G_n \mathop{\longrightarrow} \limits_{n \to +\infty} \hat{G}$ strongly in $L^2(\Omega)$.
Let us show that $\hat{G}$ is a solution of \eqref{eq:linearized discrete inside}.
By definition, one has
\[ a(\G_n; G_n , \tilde{G}_h) + \mathsf{p}(G_n,\tilde{G}_h) = 0, \quad \forall \tilde{G}_h \in W_{h0}. \]
Using the strong and weak convergences proved above, one gets
\[ a(\G_\infty; \hat{G} , \tilde{G}_h) + \mathsf{p}(\hat{G},\tilde{G}_h) = 0, \quad \forall \tilde{G}_h \in W_{h0}. \]
Using Lemma \ref{th:unique linearized}, one has $\hat{G} = G_\infty$.
Let us now prove that $G_n \mathop{\longrightarrow} \limits_{n \to +\infty} G_\infty$ strongly in $H^1(\Omega)$.
Let $\tilde{G}_h \in W_{h0}$.
By definition of $G_h$ and $G_n$, one has
\begin{align*}
0 & =  a(\G_\infty; G_\infty, \tilde{G}_h) - a(\G_n;G_n,\tilde{G}_h) + \mathsf{p}(G_\infty - G_n, \tilde{G}_h) \\
& = a(\G_\infty; G_\infty-G_n, \tilde{G}_h) + \mathsf{p}(G_\infty - G_n, \tilde{G}_h) + a(\G_\infty; G_n, \tilde{G}_h) - a(\G_n; G_n, \tilde{G}_h). 
\end{align*}
Thus,
\[ a(\G_\infty; G_\infty-G_n, \tilde{G}_h) + \mathsf{p}(G_\infty - G_n, \tilde{G}_h) = - a(\G_\infty; G_n, \tilde{G}_h) - \mathsf{p}(G_n, \tilde{G}_h) \mathop{\longrightarrow}_{n \to +\infty} 0,  \]
as $G_n \rightharpoonup G_\infty$ weakly in $H^1(\Omega)$.
Let $F := \lim_{n \to +\infty} G_\infty - G_n \in W_{h0}$.
As $a(\G_\infty)$ and $\mathsf{p}$ are continuous, one has
\[ a(\G_\infty; F, \tilde{G}_h) + \mathsf{p}(F, \tilde{G}_h) = 0, \quad \forall \tilde{G}_h \in W_{h0}. \] 
We now take $\tilde{G}_h = F$ and use the coercivity shown in the proof of Lemma \ref{th:unique linearized} and get
\[ C \Vert F \Vert_{H^1(\Omega)}^2 \le a(\G_\infty; F, F) + \mathsf{p}(F, F) = 0. \]
We thus obtain the expected strong convergence in $H^1(\Omega)$.
The existence of the fixed point is provided by the Brouwer fixed point theorem, see \cite{brezis}.
\end{proof}

\subsection{Recomputing $\varphi_h$}
Following \cite{gallistl2017variational}, the approximation space is written as
\[ V_h := \left\{ \psi_h \in \mathbb{P}^2(\mathcal{T}_h)^3 ; \int_\Omega \psi_h = 0 \right\}.\]
A least-squares method is used.
The bilinear form $\mathtt{a} \in \mathcal{L}(V_h^2)$ is defined as
\[ \mathtt{a}(\varphi_h, \tilde{\varphi}_h) := \int_\Omega \nabla \varphi_h \cdot \nabla \tilde{\varphi}_h. \]
It is classically coercive over $V_h$.
The linear form $\mathtt{b} \in \mathcal{L}(V_h)$ is defined as
\[ \mathtt{b}(\tilde{\varphi}_h) := \int_\Omega \G_h\cdot \nabla \tilde{\varphi}_h. \]
\begin{lemma}
There exists a unique $\varphi_h \in V_h$ solution of
\begin{equation}
\label{eq:discrete solution}
\mathtt{a}(\varphi_h, \tilde{\varphi}_h) = \mathtt{b}(\tilde{\varphi}_h), \quad \forall \tilde{\varphi}_h \in V_h.
\end{equation}
\end{lemma}
The proof is omitted for concision, as this is a classical result.

\subsection{Convergence of the FEM}
\begin{theorem}
\label{th:convergence rate}
Let $(\G_h)_h$ be a sequence of solutions of \eqref{eq:discrete eq}.
We define $(\varphi_h)_h$ as the sequence of unique solutions of \eqref{eq:discrete solution} associated to $(\G_h)_h$.
There exists $\varphi \in \mathcal{V}$, solution of \eqref{eq:cut-off main eq phi}, such that, up to a subsequence,
\[ \lim_{h \to 0} \Vert \varphi - \varphi_h \Vert_{H^1(\Omega)} = 0. \]
\end{theorem}

\begin{proof}
In the proof of Proposition \ref{th:nonlinear discrete eq}, we have shown that there exists $C > 0$, independent of $h > 0$, such that
\[ \Vert \G_h \Vert_{H^1(\Omega)} \le C \Vert G_D \Vert_{H^\frac12(\partial \Omega)}. \]
By compactness, there exists $\G \in W_D$ such that, up to a subsequence,
\[ \G_h \mathop{\rightharpoonup}_{h \to 0} \G \text{ weakly in } H^1(\Omega). \]
Note that due to Rellich--Kondrachov theorem \cite[Theorem 9.16]{brezis}, one also has, up to a subsequence,
\[ \G_h \mathop{\longrightarrow}_{h \to 0} \G \text{ strongly in } L^2(\Omega). \]
We now show that $\G$ verifies \eqref{eq:gradient eq}.
Let $\Phi \in \mathcal{C}^\infty_c(\Omega)^{3 \times 2}$.
One has
\begin{align*}
\mathsf{p}(\G_h,\mathcal{I}_h \Phi) &= \int_\Omega (\G_{h,y}^x - \G_{h,x}^y) \cdot ((\mathcal{I}_h \Phi^x)_y - (\mathcal{I}_h \Phi^y)_{x}) \\
& \mathop{\longrightarrow}_{h \to 0} \int_\Omega (\G_{y}^x - \G_{x}^y) \cdot (\Phi^x_y - \Phi^y_{x}),
\end{align*}
as $\Phi^x_y - \Phi^y_{x} \in \mathcal{C}^\infty_c(\Omega)^3$ and owing to the weak convergence of $\G_h \rightharpoonup \G$ in $H^1(\Omega)$.
We now test \eqref{eq:linearized discrete inside} with $\mathcal{I}_h \Phi$,
\[ 0 = a(\G_h; \G_h, \mathcal{I}_h \Phi) + \mathsf{p}(\G_h, \mathcal{I}_h \Phi) \mathop{\longrightarrow}_{h \to 0} a(\G;\G,\Phi) + \mathsf{p}(\G, \Phi), \]
owing to the weak convergence of $\G_h \rightharpoonup \G$ in $H^1(\Omega)$, and the strong convergence of $\G_h \to \G$ in $L^2(\Omega)$.
As $G_D \in \mathbb{H}^{\frac12}(\partial \Omega)^{3 \times 2}$, there exists $\hat{G}_D \in H^1(\Omega)^{3 \times 2}$, $\hat{G}_D = G_D$ on $\partial \Omega$ and $\hat{G}_{D,y}^x = \hat{G}_{D,x}^y$.
We define $\hat{\G} := \G -  \hat{G}_D \in H^1_0(\Omega)$ and $f := -\bar{A}(\G) \hat{G}_D \in L^2(\Omega)$.
By density, one has
\[ a(\G;\hat{\G}, \hat{\G}) + \mathsf{p}(\hat{\G}, \hat{\G}) = \int_\Omega f \cdot \bar{A}(\G) \hat{\G}.  \] 
Therefore, $\hat{\G}$ minimizes over $\tilde{G} \in H^1_0(\Omega)$,
\[ \Vert \bar{A}(\G) \tilde{G} - f \Vert^2_{L^2(\Omega)} + \eta \Vert \tilde{G}^x_y -  \tilde{G}^y_x \Vert^2_{L^2(\Omega)}, \]
and thus $\G^x_y = \G^y_x$ a.e. in $\Omega$ and $\bar{A}(\G) \G = 0$ a.e. in $\Omega$.
Finally, $\G \in \mathbb{W}_D$ is a solution of \eqref{eq:gradient eq}.

We define $\varphi$ from $\G$ as in Proposition \ref{th:phi}.
Using a Poincar\'e inequality, there exists $C>0$,
\begin{align*}
\Vert \varphi - \varphi_h \Vert_{H^1(\Omega)} &\le 
C \Vert \nabla \varphi - \nabla \varphi_h \Vert_{L^2(\Omega)}, \\
& \le C\Vert \nabla \varphi - \G \Vert_{L^2(\Omega)} + C\Vert \G - \G_h \Vert_{L^2(\Omega)} + C\Vert \G_h - \nabla \varphi_h \Vert_{L^2(\Omega)}, \\
&= C\Vert \G - \G_h \Vert_{L^2(\Omega)},
\end{align*}
as $\nabla \varphi = \G$, and $\nabla \varphi_h = \G_h$ in $L^2(\Omega)$.
We can now conclude that, up to a subsequence,
\[ \Vert \varphi - \varphi_h \Vert_{H^1(\Omega)} \le C \Vert \G - \G_h \Vert_{L^2(\Omega)} \mathop{\longrightarrow}_{h \to 0} 0. \]
\end{proof}

\begin{remark}[Convergence rate]
\label{rk:convergence rate}
The computation performed in the proof of Theorem \ref{th:convergence rate} shows that $\Vert \varphi - \varphi_h \Vert_{H^1(\Omega)}$ is controlled by $\Vert \mathcal{G} - \mathcal{G}_h \Vert_{L^2(\Omega)}$.
As the bilinear form $a(\G)$ is symmetric, one can wonder if results along the lines of an Aubin--Nitsche trick (in a nonlinear setting),
see \cite{dobrowolski1980finite} for instance, could lead to a convergence rate of order two in $h$.
Numerically, we observe it in Section \ref{sec:num examp}.
\end{remark}

\section{Numerical examples}
\label{sec:num examp}
This paper uses Firedrake \cite{firedrake} to solve \eqref{eq:discrete eq} with a Newton method.
The boundary condition $G_h = \mathcal{I}_h G_D$ on $\partial \Omega$ is imposed strongly.
The stopping criterion based on the relative residual in $H^1$-norm is set to $\varepsilon := 10^{-8}$.
The penalty parameter for $\mathsf{p}$ is taken as $\eta := 1$, unless otherwise specified.
The value of the penalty parameter has been determined by trial and error.
A value of $\eta$ too small impacts the resolution of \eqref{eq:linearized discrete inside} for a lack of coercivity.
On the contrary, a value of $\eta$ too large impacts the convergence of the Newton method as the penalty term becomes dominant in \eqref{eq:discrete eq}.
An initial guess $\G_{h,0} \in W_{hD}$ for the Newton method is computed as the solution of
\[\int_\Omega \nabla \G_{h,0} \cdot \nabla \tilde{G}_h, \quad \forall \tilde{G}_h \in W_{h0}. \]

\subsection{Hyperboloid}
\label{ex:hyper}
This test case comes from \cite{lebee2018fitting}.
The reference solution is
\[\phi(x,y) = \left(\rho(x)\cos(\alpha y), \rho(x) \sin(\alpha y), z(x) \right)^\mathsf{T}, 
\text{ where } 
\left\{ \begin{aligned}
\rho(x) &= \sqrt{4 c_0^2 x^2 + 1}, \\
z(x) &= 2 s_0 x, \\
\alpha &= \left(1 - s_0^2 \right)^{-1/2}, 
\end{aligned} \right. \]
$c_0 = \cos(\frac\theta{2})$, $s_0 = \sin\left(\frac\theta{2}\right)$ and $\theta \in \left(0, \frac{2\pi}3\right)$.
The domain is $\Omega = (-s_0^*,s_0^*) \times (0, \frac{2\pi}\alpha)$, where $s_0^* = \sin\left(0.5\cos^{-1}\left(0.5\cos\left(\frac{\theta}2\right)^{-1}\right)\right)$.
Structured meshes, periodic in y, are used.
This translates into the fact that the dofs on the lines of equation $y=0$ and $y=\frac{2\pi}\alpha$ of $\partial \Omega$ are one and the same, but remain unknown.
$\nabla \phi$ is used as gradient Dirichlet boundary condition on the lines of equations $x=-s_0^*$ and $x=s_0^*$.
A convergence test is performed for $\theta = \frac\pi{2}$.
Table \ref{tab:convergence rate} contains the errors and estimated convergence rates.
\begin{table}[!htp]
\centering
\begin{tabular}{|c|c|c|c|c|c|c|}
\hline
$h$ & nb dofs & nb iterations & $\Vert \G_h - \G\Vert_{H^1}$ & rate & $\Vert \G_h - \G\Vert_{L^2}$ & rate \\ \hline
4.4e-01 & 1,260 & 7 & 3.128e+00 & - & 7.181e-01 & - \\ \hline
2.2e-01 & 4,920 & 3 & 5.028e-01 & 2.68 & 6.284e-02 & 3.57 \\ \hline
1.1e-01 & 19,440 & 3 & 2.347e-01 & 1.11 & 1.491e-02 & 2.09  \\ \hline
5.6e-02 & 77,280 & 3 & 1.156e-01 & 1.03 & 3.686e-03 & 2.03  \\ \hline
2.8e-02 & 308,160 & 3 & 5.757e-02 & 1.01 & 9.190e-04 & 2.01 \\ \hline
1.4e-02 & 1,230,720 & 3 & 2.876e-02 & 1.00 & 2.296e-04 & 2.01 \\ \hline
\end{tabular}
\caption{Hyperboloid: estimated convergence rate and number of Newton iterations.}
\label{tab:convergence rate}
\end{table}
The convergence rate is estimated using the formula \[ 2 \log\left(\frac{e_1}{e_2}\right)  \log\left(\frac{\mathrm{card}(\mathcal{T}_{h_1})}{\mathrm{card}(\mathcal{T}_{h_2})}\right)^{-1}, \]
where $e_1$ and $e_2$ are the errors. 
The number of Newton iterations is also reported.
The results in Table \ref{tab:convergence rate} are compatible with the second order rate announced in Remark \ref{rk:convergence rate}.

\subsection{Annulus} 
\label{sec:annulus}
The domain is $\Omega = \left(0, \frac34\right) \times (0, 2\pi)$.
A structured mesh, periodic in the $y$ direction, is used to mesh $\Omega$.
The mesh has a size $h = 0.251$ and contains $45,900$ dofs.
The gradient Dirichlet boundary conditions are
\[ \left\{ \begin{aligned}
& G_D^x = kx e_r, \\
& G_D^y = \frac4{4-|G_D^x|^2} e_{\theta}, \\
\end{aligned} \right. \]
where $(e_r,e_\theta)$ is the polar basis of $\mathbb{R}^2$ and $k$ is a constant.
We consider two different cases: $k=2.21$ and $k=1.5$.
The resulting surfaces are presented in Figure \ref{fig:annulus}.
\begin{figure}[htp!]
     \centering
     \begin{subfigure}[b]{0.45\textwidth}
         \centering
         \includegraphics[scale=.3]{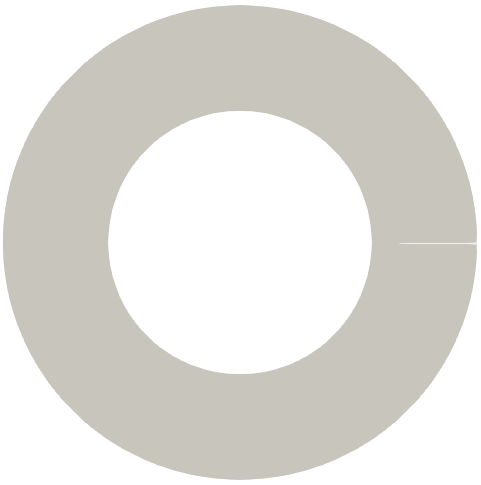}
         \caption{$k=2.21$}
     \end{subfigure}
     \hfill
     \begin{subfigure}[b]{0.45\textwidth}
         \centering
         \includegraphics[scale=.3]{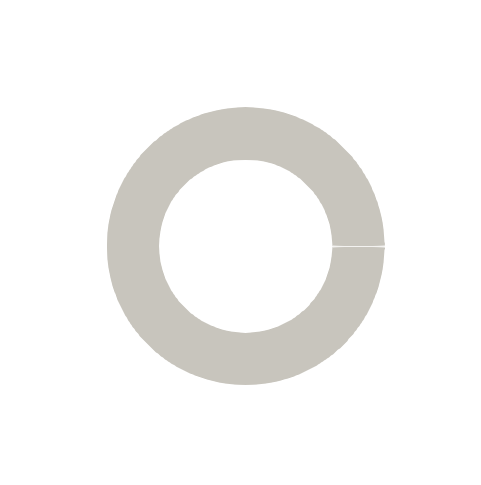}
         \caption{$k=1.5$}
     \end{subfigure}
\caption{Annulus: Computed surfaces.}
\label{fig:annulus}
\end{figure}
Figure \ref{fig:annulus constraint} presents $|\mathcal{G}_{h,y}|$ for both cases.
$|\G_{h,x}|$ is not reported since the inequality constraints \eqref{eq:ineq constraint} on it are verified for both computations.
\begin{figure}[htp!]
     \centering
     \begin{subfigure}[b]{0.45\textwidth}
         \centering
         \includegraphics[scale=.3]{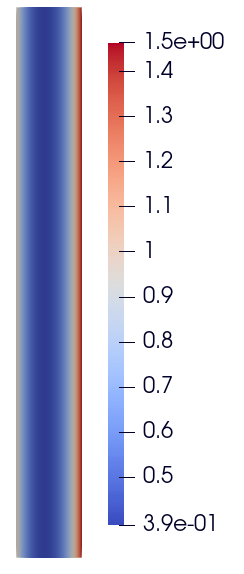}
         \caption{$k=2.21$}
     \end{subfigure}
     \hfill
     \begin{subfigure}[b]{0.45\textwidth}
         \centering
         \includegraphics[scale=.3]{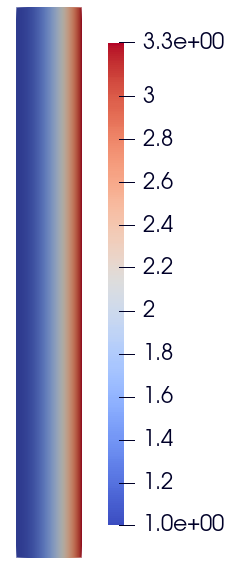}
         \caption{$k=1.5$}
     \end{subfigure}
     \caption{Annulus: $|\G_{h,y}|^2$.}
\label{fig:annulus constraint}
\end{figure}
One can observe that for $k=2.21$, $\Omega' = \Omega$ whereas for $k=1.5$, $\Omega' \subsetneq \Omega$.
Indeed, in a large band in the middle of the domain, one can notice that $|\G_{h,y}| < 1$.
This is confirmed by looking at $v$, see Figure \ref{fig:annulus v}.
\begin{figure}[htp!]
     \centering
     \begin{subfigure}[b]{0.45\textwidth}
         \centering
         \includegraphics[scale=.3]{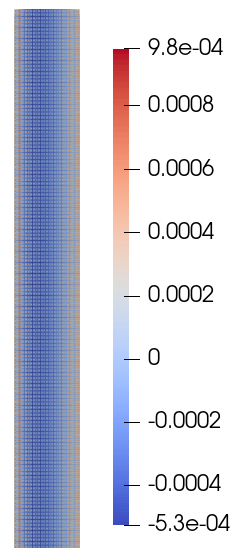}
         \caption{$k=2.21$}
     \end{subfigure}
     \hfill
     \begin{subfigure}[b]{0.45\textwidth}
         \centering
         \includegraphics[scale=.3]{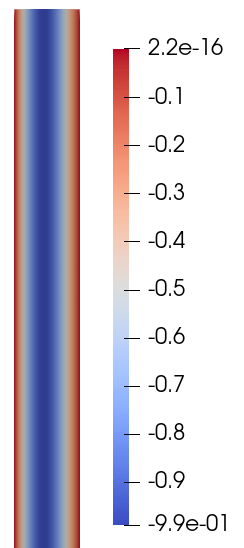}
         \caption{$k=1.5$}
     \end{subfigure}
     \caption{Annulus: $v$.}
\label{fig:annulus v}
\end{figure}
The violation of the constraints in $\Omega \setminus \Omega'$ is interpreted as being the result of having gradient Dirichlet boundary conditions on all of $\partial \Omega$.
Indeed, as stated previously, the constraints are not verified after the pattern fully folds on a curve, which corresponds to $|\varphi_y| = 1$.
As gradient Dirichlet boundary conditions are imposed on all of $\partial \Omega$, the pattern cannot relax in any part of the domain, as would have been possible with a part of the boundary having stress-free boundary conditions.
This leads to an over-constrained pattern and thus to a solution that is not physical in some parts of the domain.

\subsection{Axysymmetric surface}
This example comes from \cite{lebee2018fitting}.
The domain is $(0,2\pi) \times (0,4)$.
A structured mesh, periodic in the $x$ direction, is used.
The mesh has a size $h = 0.126$ and contains $135,750$ dofs.
The reference solution writes 
\[ \phi(x,y) := (\rho(y) \cos(x), \rho(y) \sin(x), z(x)), \]
where $\rho$ is a solution of
\[ \ddot{\rho} = \frac{4 \rho}{(4 - \rho^2)^2}. \]
An explicit Runge--Kutta method of order 5 (see \cite{hairer2006geometric}) is used to integrate this ODE with initial conditions $\rho_0 = 0.1$ and $\dot{\rho}_0 = 0$.
Because the reference solution is not know analytically, the gradient Dirichlet boundary condition is imposed weakly using a least-squares penalty in order to remove numerical issues.
Let $\mathcal{E}_h$ the set of the edges of $\mathcal{T}_h$.
The set $\mathcal{E}_h$ is partitioned as $\mathcal{E}_h := \mathcal{E}_h^i \cup \mathcal{E}_h^b$, where for all $e \in \mathcal{E}_h^b$, $e \subset \partial \Omega$ and $\cup_{e \in \mathcal{E}^b_h} {e} = \partial \Omega$.
Let $h_e := \mathrm{diam}(e)$, where $e \in \mathcal{E}_h$.
Therefore, we search for $\G_h \in W_h$,
\[ a(\G_{h}; \G_{h}, \tilde{G}_h) + \sum_{e \in \mathcal{E}^b_h} \frac\eta{h_e} \int_e \G_h : \tilde{G}_h =  \sum_{e \in \mathcal{E}^b_h} \frac\eta{h_e} \int_e \mathcal{I}_h G_D : \tilde{G}_h , \quad \forall \tilde{G}_h \in W_{h0}. \]
\begin{figure}[!htp]
\centering
\includegraphics[scale=0.25]{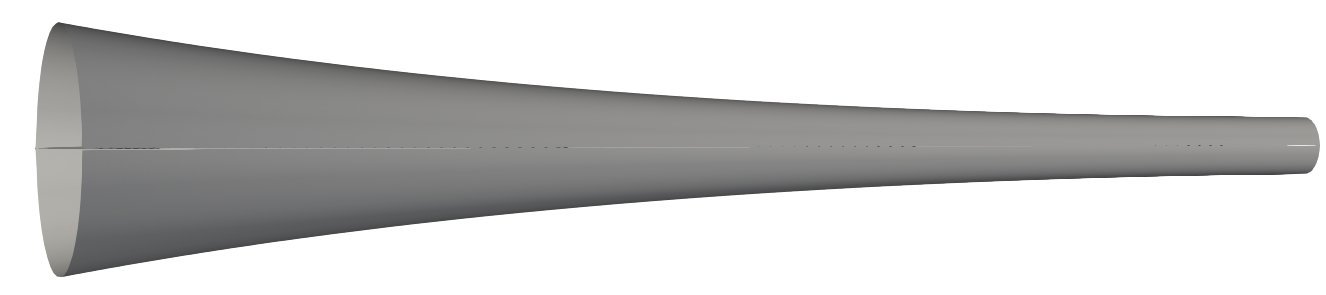}
\caption{Axysymmetric surface: Computed surface.}
\label{fig:axy shape}
\end{figure}
Figure \ref{fig:axy shape} shows the computed surface, which looks as expected in \cite{lebee2018fitting}.

\subsection{Deformed hyperboloid}
\label{sec:deformed}
The domain $\Omega$ is the same as in Section \ref{ex:hyper}.
A structured mesh, periodic in the $y$ direction, is used to mesh $\Omega$.
The mesh has a size $h = 0.044$ and contains $120,600$ dofs.
The gradient Dirichlet boundary conditions are the same as in Section \ref{ex:hyper} for the left-hand side of $\Omega$ and rotated by an angle $\frac\pi3$, with respects to Section \ref{ex:hyper}, on the right-hand side of $\Omega$.
This breaks the axial symmetry that allows one to reduce \eqref{eq:min surface eq} to an ODE.
The penalty parameter is taken as $\eta=10$ for this computation.
The resulting surface is presented in Figure \ref{fig:def hyper}.
\begin{figure}[htp!]
     \centering
     \begin{subfigure}[b]{0.45\textwidth}
         \centering
         \includegraphics[scale=.3]{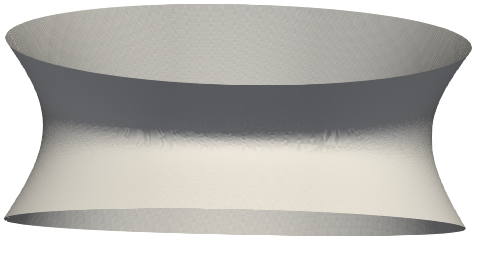}
     \end{subfigure}
     \hfill
     \begin{subfigure}[b]{0.45\textwidth}
         \centering
         \includegraphics[scale=.3]{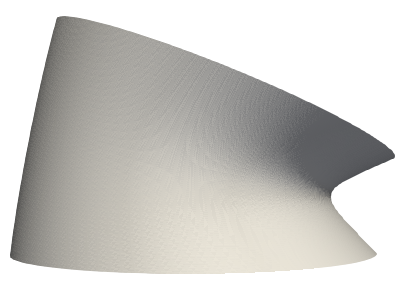}
     \end{subfigure}
     \caption{Deformed hyperboloid: Computed surface.}
\label{fig:def hyper}
\end{figure}
Note that the resulting configuration is quite `far' from the original hyperboloid.
The gradients are reported in Figure \ref{fig:def hyper grad}.
\begin{figure}[htp!]
     \centering
     \begin{subfigure}[b]{0.45\textwidth}
         \centering
         \includegraphics[scale=0.3]{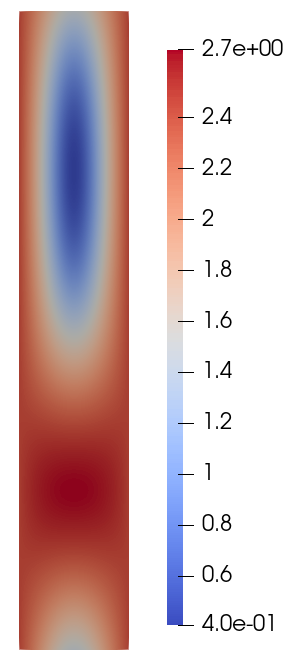}
         \caption{$|\G_{h,x}|^2$}
     \end{subfigure}
     \hfill
     \begin{subfigure}[b]{0.45\textwidth}
         \centering
         \includegraphics[scale=0.3]{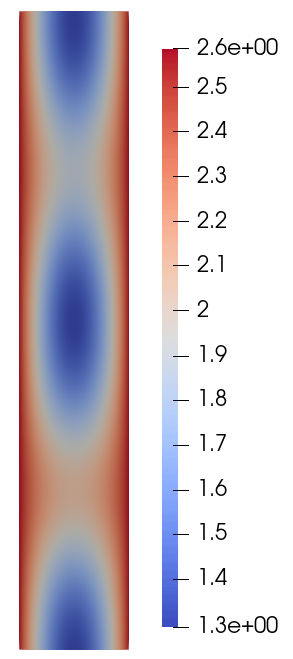}
         \caption{$|\G_{h,y}|^2$}
     \end{subfigure}
     \caption{Deformed hyperboloid.}
\label{fig:def hyper grad}
\end{figure}
Figure \ref{fig:def hyper grad} shows that the inequality constraints \eqref{eq:ineq constraint} are verified in all of $\Omega$ which suggests that \eqref{eq:local basis} and \eqref{eq:microstructure} will also be verified in all of $\Omega$ because of Theorem \ref{th:general}.
This is confirmed by looking at Figure \ref{fig:def hyper constraint}.
\begin{figure}[htp!]
     \centering
     \begin{subfigure}[b]{0.45\textwidth}
         \centering
         \includegraphics[scale=.3]{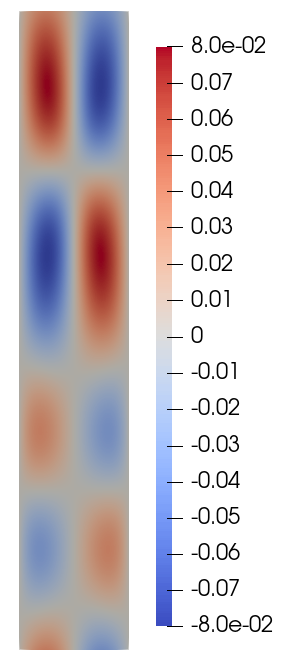}
         \caption{$u$}
     \end{subfigure}
     \hfill
     \begin{subfigure}[b]{0.45\textwidth}
         \centering
         \includegraphics[scale=.3]{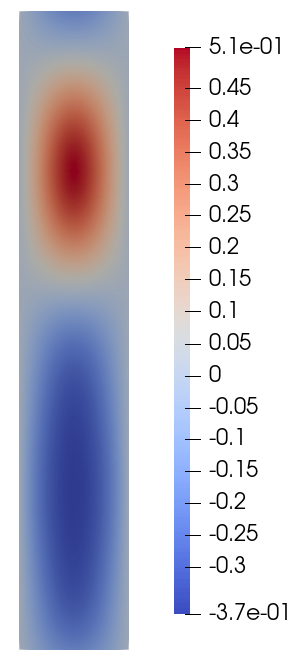}
         \caption{$v$}
     \end{subfigure}
     \caption{Deformed hyperboloid.}
\label{fig:def hyper constraint}
\end{figure}
One can see that $u$ and $v$ are small in all of $\Omega$.

\section{Conclusion}
In this paper, the existence of solutions to the constrained system of equations describing a Miura surface are proved under some hypotheses.
Then, a numerical method based on a stabilized formulation, $\mathbb{P}^1$--Lagrange elements, and a Newton method is introduced to approximate Miura surfaces.
The method is proved to converge in $H^1$-norm in the space discretization parameter $h$ and a convergence order of two is observed in practice.
Some numerical tests are performed and show the robustness of the method.
Future work includes studying the constrained nonlinear hyperbolic PDE derived by homogenizing the eggbox pattern, as in \cite{nassar2017curvature}.

\section*{Code availability}
The code is available at \url{https://github.com/marazzaf/Miura.git}

\section*{Acknowledgment}
The author would like to thank Hussein Nassar from University of Missouri for stimulating discussions that lead to the content of this paper.
The author would also like to thank Zhaonan Dong and Alexandre Ern from Inria Paris for stimulating discussions.
Finally, the author would like to thank the patient referees for their careful and informative comments.

\section*{Funding}
This work is supported by the US National Science Foundation under grant number OIA-1946231 and the Louisiana Board of Regents for the Louisiana Materials Design Alliance (LAMDA).

\bibliographystyle{plain}
\bibliography{bib}

\end{document}